\documentclass[12pt]{amsart}
\usepackage{amssymb,amsmath,mathrsfs}
\usepackage[ps2pdf,colorlinks=true,urlcolor=blue,
citecolor=red,linkcolor=blue,linktocpage,pdfpagelabels,bookmarksnumbered,bookmarksopen]{hyperref}
\usepackage{epsfig,graphicx,color,mathrsfs}
\usepackage[english]{babel}

\usepackage[left=2.8cm,right=2.8cm,top=2.5cm,bottom=2.5cm]{geometry}

\newcommand{\N}{{\mathbb N}}

\newcommand{\R}{{\mathbb R}}

\newcommand{\be}{\begin{equation}}
\newcommand{\ee}{\end{equation}}

\newcommand{\eps}{\varepsilon}

\numberwithin{equation}{section}
\newtheorem{theorem}{Theorem}[section]
\newtheorem{proposition}[theorem]{Proposition}
\newtheorem{corollary}[theorem]{Corollary}
\newtheorem{lemma}[theorem]{Lemma}
\newtheorem{definition}[theorem]{Definition}
\theoremstyle{definition}
\newtheorem{remark}[theorem]{Remark}

\newcommand{\brm}{\begin{remark}\rm}
\newcommand{\erm}{\end{remark}}
\newcommand{\brms}{\begin{remark}\rm}
\newcommand{\erms}{\end{remark}}
\newcommand{\bte}{\begin{theorem}}
\newcommand{\ete}{\end{theorem}}
\newcommand{\bpr}{\begin{proposition}}
\newcommand{\epr}{\end{proposition}}
\newcommand{\ble}{\begin{lemma}}
\newcommand{\ele}{\end{lemma}}
\newcommand{\beq}{\begin{equation}}
\newcommand{\eeq}{\end{equation}}
\newcommand{\bdm}{\begin{displaymath}}
\newcommand{\edm}{\end{displaymath}}
\numberwithin{equation}{section}

\newcommand{\bos}{\begin{remark}\rm}
\newcommand{\eos}{\end{remark}}

\newcommand{\ben}{\begin{enumerate}}
\newcommand{\een}{\end{enumerate}}

\title[Generalized Polya-Szeg\"o inequality and applications]{Generalized 
Polya-Szeg\"o inequality and applications to some quasi-linear elliptic problems}

\author[H.\ Hajaiej]{Hichem Hajaiej}
\author[M.\ Squassina]{Marco Squassina}

\address{Hichem Hajaiej
\newline\indent
Department of Mathematics
\newline\indent
King Saud University, College of Sciences
\newline\indent
P.O. Box 2455, 11451 Riyadh, Saudi Arabia}
\email{hichem.hajaiej@gmail.com}

\address{Marco Squassina
\newline\indent
Department of Computer Science
\newline\indent
University of Verona
\newline\indent
C\'a Vignal 2, Strada Le Grazie 15, I-37134 Verona, Italy}
\email{marco.squassina@univr.it}

\thanks{The second author was supported by the 2007 MIUR national research
project entitled {\em ``Variational and Topological Methods in the Study of
Nonlinear Phenomena''}}
\begin{document}

\begin{abstract}
We generalize Polya-Szeg\"o inequality to integrands depending on $u$ and its
gradient. Under additional assumptions, we establish equality cases in this
generalized inequality. We also give relevant applications of our study to a class of quasi-linear
elliptic equations and systems.
\end{abstract}
\maketitle



\section{Introduction}

The Polya-Szeg\"o inequality asserts that the $L^2$ norm of the gradient of a positive function $u$
in $W^{1,p}(\R^N)$ cannot increase under Schwarz symmetrization,
\begin{equation}
	\label{classic}
\int_{\R^N}|\nabla u^*|^2dx\leq \int_{\R^N}|\nabla u|^2dx.
\end{equation}
The Schwarz rearrangement of $u$ is denoted here by $u^*$. Inequality~\eqref{classic} has numerous 
applications in physics. It was first used in 1945 by G.\ Polya and G. Szeg\"o
to prove that the capacity of a condenser diminishes or remains unchanged by 
applying the process of Schwarz symmetrization (see~\cite{polya1}).
Inequality~\eqref{classic} was also the key ingredients to show that, among
all bounded bodies with fixed measure, balls have the minimal capacity (see~\cite[Theorem 11.17]{liebloss}).
Finally~\eqref{classic} has also played a crucial role in the solution of the famous
Choquard's conjecture (see~\cite{lieb}). It is heavily connected to the isoperimetric inequality
and to Riesz-type rearrangement inequalities. Moreover, it turned out that~\eqref{classic} 
is extremely helpful in establishing the existence of ground states solutions of the nonlinear
Schr\"odinger equation
\begin{equation}
	\label{Sch}
\begin{cases}
{\rm i}\partial_t\Phi+\Delta\Phi+f(|x|,\Phi)=0  & \text{in $\R^N\times(0,\infty)$},\\
\Phi(x,0)=\Phi_0(x)  & \text{in $\R^N$}.
\end{cases}
\end{equation}
A ground state solution of equation~\eqref{Sch} is a positive  
solution to the following associated variational problem
\begin{equation}
\label{varPbSc}
\inf\left\{\frac{1}{2}\int_{\R^N}|\nabla u|^2dx-\int_{\R^N}F(|x|,u)dx:\, u\in H^1(\R^N),\,\,\|u\|_{L^2}=1\right\},
\end{equation}
where $F(|x|,s)$ is the primitive of $f(|x|,\cdot)$ with $F(|x|,0)=0$. Inequality~\eqref{classic}
together with the generalized Hardy-Littlewood inequality were crucial to prove that~\eqref{varPbSc}
admits a radial and radially decreasing solution.\ Furthermore, under appropriate regularity
assumptions on the nonlinearity $F$, there exists a Lagrange multiplier $\lambda$ such that any minimizer 
of~\eqref{varPbSc} is a solution of the following semi-linear elliptic PDE
$$
-\Delta u+f(|x|,u)+\lambda u=0,\quad\text{in $\R^N$}.
$$
We refer the reader to~\cite{hajstuampa} for a detailed analysis. The same approach applies 
to the more general quasi-linear PDE
$$
-\Delta_p u+f(|x|,u)+\lambda u=0,\quad\text{in $\R^N$}.
$$
where $\Delta_p u$ means ${\rm div}(|\nabla u|^{p-2}\nabla u)$, and we can derive similar 
properties of ground state solutions since~\eqref{classic} extends to gradients that
are in $L^p(\R^N)$ in place of $L^2(\R^N)$, namely
\begin{equation}
	\label{p-classic}
\int_{\R^N}|\nabla u^*|^pdx\leq \int_{\R^N}|\nabla u|^pdx.
\end{equation}
Due to the multitude of applications in physics, rearrangement inequalities like~\eqref{classic} and
\eqref{p-classic} have attracted a huge number of mathematicians from the middle of the
last century. Different approaches were built up to establish
these inequalities such as heat-kernel methods, slicing and cut-off techniques 
and two-point rearrangement.

A generalization of inequality~\eqref{p-classic} to suitable convex integrands $A:\R_+\to\R_+$,
\begin{equation}
	\label{A-classic}
\int_{\R^N}A(|\nabla u^*|)dx\leq \int_{\R^N} A(|\nabla u|)dx,
\end{equation}
was first established by Almgren and Lieb (see \cite{almgrenlieb}). Inequality~\eqref{A-classic}
is important in studying the continuity and discontinuity of Schwarz symmetrization
in Sobolev spaces (see e.g.\ \cite{almgrenlieb,burchard}). It also permits us to study
symmetry properties of variational problems involving integrals of type $\int_{\R^N}A(|\nabla u|)dx$.
Extensions of Polya-Szeg\"o inequality to more general operators of the form
$$
j(s,\xi)=b(s)A(|\xi|),\quad s\in\R,\,\xi\in\R^N,
$$
on bounded domains have been investigated by Kawohl, Mossino and Bandle. More precisely,
they proved that
\begin{equation}
	\label{PS-split}
\int_{\Omega^*}b(u^*)A(|\nabla u^*|)dx\leq \int_{\Omega} b(u)A(|\nabla u|)dx,
\end{equation}
where $\Omega^*$ denotes the ball in $\R^N$ centered at the origin having
the Lebesgue measure of $\Omega$, under suitably convexity, monotonicity
and growth assumptions (see e.g.\ \cite{bandle,kawohl,mossino}).
Numerous applications of~\eqref{PS-split} have been discussed in the above references.
In~\cite{tahraoui}, Tahraoui claimed that a general integrand $j(s,\xi)$ with appropriate properties
can be written in the form
$$
\sum_{i=1}^\infty b_i(s)A_i(|\xi|)+R_1(s)+R_2(\xi),\quad s\in\R,\,\xi\in\R^N,
$$
where $b_i$ and $A_i$ are such that inequality~\eqref{PS-split} holds. However, there are
some mistakes in~\cite{tahraoui} and we do not believe that this density type result holds true.
Until quite recently there were no results dealing with the generalized 
Polya-Szeg\"o inequality, namely
\begin{equation}
	\label{generalPS}
\int_{\Omega^*}j(u^*,|\nabla u^*|)dx\leq \int_{\Omega} j(u,|\nabla u|)dx.
\end{equation}
While writing down this paper we have learned about a very recent survey by 
F.\ Brock \cite{brock} who was able to prove~\eqref{generalPS} under continuity,
monotonicity, convexity and growth conditions.

Following a different approach, we prove~\eqref{generalPS}
without requiring any growth conditions on $j$. As it can be easily seen it is important to
drop these conditions to the able to cover some relevant applications.
Our approach is based upon a suitable approximation of the Schwarz symmetrized
$u^*$ of a function $u$. More precisely, if  
$(H_n)_{n\geq 1}$ is a dense sequence in the set of closed half spaces $H$ containing $0$
and $u\in L^{p}_+(\R^N)$, there exists a 
sequence $(u_n)$ consisting of iterated polarizations of the $H_n$s which
converges to $u^*$ in $L^p(\R^N)$ (see~\cite{explicit,JvS}). 
On the other hand, a straightforward computation shows that
$$
	\|\nabla u\|_{L^p(\R^N)}=\|\nabla u_0\|_{L^p(\R^N)}=\cdots=\|\nabla u_n\|_{L^p(\R^N)},
	\quad\text{for all $n\in\N$}.
$$
Combining these properties with the weak lower semicontinuity of the functional
$J(u)=\int j(u,|\nabla u|)dx$ enable us to conclude (see Theorem~\ref{gpz}). Note that~\eqref{A-classic}
was proved using coarea formula; however this approach does not apply to integrands
depending both on $u$ and its gradient since one has to apply simultaneously the coarea formula
to $|\nabla u|$ and to decompose $u$ with the Layer-Cake principle.

Detailed applications of our results concerning~\eqref{generalPS} are given 
in Section~\ref{applsection}, where we determine a suitable class of assumptions
that allow us to solve the (vector) problem of minimizing the functional 
$J:W^{1,p}(\R^N,\R^m)\to\R$, $m\geq 1$,
$$
J(u)=\sum_{k=1}^m \int_{\R^N} j_k(u_k,|\nabla u_k|)dx-\int_{\R^N} F(|x|,u_1,\dots,u_m)dx.
$$
on the constraint of functions $u=(u_1,\dots,u_m)\in W^{1,p}(\R^N,\R^m)$ such that
\begin{equation*}
G_k(u_k),\, j_k(u_k,|\nabla u_k|)\in L^1(\R^N)\quad\text{and}
\quad\sum_{k=1}^m \int_{\R^N}G_k(u_k)dx=1.
\end{equation*}
Notice that Brock's method is based on an intermediate maximization problem and cannot yield to the establishment
of equality cases. Our approximation approach was also fruitful in determining the relationship between
$u$ and $u^*$ such that
\begin{equation}
	\label{EqgeneralPS}
\int_{\R^N}j(u^*,|\nabla u^*|)dx=\int_{\R^N} j(u,|\nabla u|)dx.
\end{equation}
Indeed, under suitable assumptions, we prove that~\eqref{EqgeneralPS} yields 
$$
\int_{\R^N}|\nabla u^*|^pdx=\int_{\R^N}|\nabla u|^pdx.
$$ 
This is very useful, as for $j(\xi)=|\xi|^p$, identity cases were completely studied 
in the breakthrough paper of Brothers and Ziemer~\cite{bzim}.

\vskip15pt
\noindent
The paper is organized as follows.
\vskip4pt
\noindent
Section~\ref{prelimfacts} is dedicated to some preliminary stuff, especially the ones 
concerning the invariance of a class of functionals under polarization. 
These observations are crucial, in Section~\ref{PSineq}, to establish 
in a simple way the generalized Polya-Szeg\"o inequality.
With the help of this, we then study in Section~\ref{applsection} a class 
of variational problems involving quasi-linear operators.
We first prove that our variational problem~\eqref{minprob} always admits a Schwarz symmetric
minimizer. Then, using the result we have established in Corollary~\ref{identitcor}, 
under suitable assumptions we 
show that all minimizers $u$ of problem~\eqref{minprob} are radially symmetric and radially decreasing, 
up to a translation in $\R^N$, provided the set of critical points of $u^*$ has zero measure.
Another meaningful variant of the main application, related to a recent paper of the second author,
is also stated in Theorem~\ref{mainappl3}.
\medskip

\vskip15pt
\begin{center}\textbf{Notations.}\end{center}
\begin{enumerate}
\item For $N\in\N$, $N\geq 1$, we denote by $|\cdot|$ the euclidean norm in $\R^N$.
\item $\R_+$ (resp.\ $\R_-$) is the set of positive (resp.\ negative) real values.
\item $\mu$ denotes the Lebesgue measure in $\R^N$.
\item $M(\R^N)$ is the set of measurable functions in $\R^N$.
\item For $p>1$ we denote by $L^p(\R^N)$ the space of $f$ in $M(\R^N)$
with $\int_{\R^N}|f|^pdx<\infty$.
\item The norm $(\int_{\R^N}|f|^pdx)^{1/p}$ in $L^p(\R^N)$ is denoted by $\|\cdot\|_p$.
\item For $p>1$ we denote by $W^{1,p}(\R^N)$ the Sobolev space of functions $f$ in $L^p(\R^N)$
having generalized partial derivatives $D_if$ in $L^p(\R^N)$, for $i=1,\dots, N$.
\item $D^{1,p}(\R^N)$ is the space of measurable functions whose gradient is in $L^p(\R^N)$.
\item $L^{p}_+(\R^N)$ is the cone of positive functions of $L^{p}(\R^N)$.
\item $W^{1,p}_+(\R^N)$ is the cone of positive functions of $W^{1,p}(\R^N)$.
\item For $R>0$, $B(0,R)$ is the ball in $\R^N$ centered at zero with radius $R$.
\end{enumerate}
\medskip

\section{Preliminary stuff}
\label{prelimfacts}
In the following $H$ will design a closed half-space of $\R^N$ containing the origin, $0_{\R^N}\in H$.
We denote by ${\mathcal H}$ the set of closed half-spaces of $\R^N$ containing the origin. We shall
equip ${\mathcal H}$ with a topology ensuring that $H_n\to H$ as $n\to\infty$ if there is a sequence of
isometries $i_n:\R^N\to\R^N$ such that $H_n=i_n(H)$ and $i_n$ converges to the identity as $n\to\infty$.
\vskip4pt
We first recall some basic notions. For more details, we refer the reader to~\cite{burchhaj}.

\begin{definition}
	A reflection $\sigma:\R^N\to\R^N$ with respect to $H$
	is an isometry such that the following properties hold
	\begin{enumerate}
		\item $\sigma\circ\sigma (x)=x$, for all $x\in\R^N$;
		\item the fixed point set of $\sigma$ separates $\R^N$ in $H$ and $\R^N\setminus H$ (interchanged by $\sigma$);
		\item $|x-y|<|x-\sigma(y)|$, for all $x,y\in H$.
	\end{enumerate}
	Given $x\in\R^N$, the reflected point $\sigma_H(x)$ will also be denoted by $x^H$.
\end{definition}

\begin{definition}
Let $H$ be a given half-space in $\R^N$.
The two-point rearrangement (or polarization) of a nonnegative real valued function
$u:\R^N\to\R_+$ with respect to a given reflection $\sigma_H$ (with respect to $H$)
is defined as 
$$
u^H(x):=
\begin{cases}
	\max\{u(x),u(\sigma_H(x))\}, & \text{for $x\in H$}, \\
	\min\{u(x),u(\sigma_H(x))\}, & \text{for $x\in \R^N\setminus H$}.
\end{cases}
$$
\end{definition}

\begin{definition}
	We say that a nonnegative measurable function $u$ is symmetrizable if
$\mu(\{x\in\R^N: u(x)>t\})<\infty$ for all $t>0$. The space of symmetrizable 
functions is denoted by $F_N$ and, of course, $L^p_+(\R^N)\subset F_N$.
Also, two functions $u,v$ are said to be equimeasurable (and we shall write $u\sim v$) when 
$$
\mu(\{x\in\R^N: u(x)>t\})=\mu(\{x\in\R^N: v(x)>t\}),
$$ 
for all $t>0$. 
\end{definition}

\begin{definition}
	For a given $u$ in $F_N$, the Schwarz symmetrization $u^*$ of $u$ is the unique
	function with the following properties (see e.g.~\cite{hajstu})
	\begin{enumerate}
		\item $u$ and $u^*$ are equimeasurable;
		\item $u^*(x)=h(|x|)$, where $h:(0,\infty)\to\R_+$ is a continuous and decreasing function.
	\end{enumerate}
	In particular, $u$, $u^H$ and $u^*$ are all equimeasurable functions (see e.g.~\cite{bae}).
	\end{definition}

\begin{lemma}
	\label{concretelem}
	Let $u\in W^{1,p}_+(\R^N)$ and let $H$ be a given half-space with $0\in H$. 
	Then $u^H\in W^{1,p}_+(\R^N)$ and, setting
	$$
	v(x):=u(x^H),\quad w(x):=u^H(x^H),\qquad x\in\R^N,
	$$
	the following facts hold:
	\begin{enumerate}
		\item We have
	\begin{align*}
		\nabla u^H(x)&=
		\begin{cases}
			\nabla u(x) & \text{for $x\in \{u>v\}\cap H$}, \\
			\nabla v(x) & \text{for $x\in \{u\leq v\}\cap H$}, \\
		\end{cases} \\
		\nabla w(x)&=
		\begin{cases}
			\nabla v(x) & \text{for $x\in \{u>v\}\cap H$}, \\
			\nabla u(x) & \text{for $x\in \{u\leq v\}\cap H$}. \\
		\end{cases}
	\end{align*}
	\item Let $j:\R_+\times \R_+\to\R_+$ be a continuous function. Then, if
	$j(u,|\nabla u|)\in L^1(\R^N)$, it follows that $j(u^H,|\nabla u^H|)\in L^1(\R^N)$ and
 \begin{equation}
	\label{fundineq}
		\int_{\R^N}j(u,|\nabla u|)dx=\int_{\R^N}j(u^H,|\nabla u^H|)dx.
	\end{equation}
\end{enumerate}
\end{lemma}
\begin{proof}
	Observe that, for all $x\in H$, we have
	$$
	u^H(x)=v(x)+(u(x)-v(x))^+,\qquad
	w(x)=u(x)-(u(x)-v(x))^+.
	$$
	Moreover, it follows that the functions $u^H,v,w$ belong to $W^{1,p}_+(\R^N)$ 
	(see~\cite[Proposition 2.3]{smewil}). Assertion (1) follows by a simple direct computation.
	Concerning (2), assume that $j(u,|\nabla u|)\in L^1(\R^N)$. Writing $\sigma_H$ 
    as $\sigma_H(x)=x_0+Rx$, where $R$ is an orthogonal linear transformation (symmetric, as reflection),
	taking into account that $|{\rm det}\,R|=1$ and
	$$
	|\nabla v(x)|=|\nabla (u(\sigma_H(x)))|=|R(\nabla u(\sigma_H(x)))|=|(\nabla u)(\sigma_H(x))|, 
	$$ 
	we have, by a change of variable, 
	\begin{align*}
		\int_{\R^N}j(u,|\nabla u|)dx & =\int_{H}j(u,|\nabla u|)dx+\int_{\R^N\setminus H}j(u,|\nabla u|)dx \\
&		=\int_{H}j(u,|\nabla u|)dx+\int_{H}j(u(\sigma_H(x)),|(\nabla u)(\sigma_H(x))|)dx \\
&		=\int_{H}j(u,|\nabla u|)dx+\int_{H}j(v,|\nabla v|)dx.
	\end{align*}
	In particular, $j(v,|\nabla v|)\in L^1(H)$. In a similar fashion, we have
	\begin{align*}
		\int_{\R^N}j(u^H,|\nabla u^H|)dx & =\int_{H}j(u^H,|\nabla u^H|)dx+\int_{H}j(u^H(\sigma_H(x)),|(\nabla u^H)(\sigma_H(x))|)dx \\
&		=\int_{H}j(u^H,|\nabla u^H|)dx+\int_{H}j(w,|\nabla w|)dx \\
&		=\int_{\{u>v\}\cap H}j(u,|\nabla u|)dx+\int_{\{u>v\}\cap H}j(v,|\nabla v|)dx \\
&		+\int_{\{u\leq v\}\cap H}j(v,|\nabla v|)dx+\int_{\{u\leq v\}\cap H}j(u,|\nabla u|)dx \\
&		=\int_{H}j(u,|\nabla u|)dx+\int_{H}j(v,|\nabla v|)dx.
	\end{align*}
	Hence $j(u^H,|\nabla u^H|)\in L^1(\R^N)$, and we have the desired identity, concluding the proof.
\end{proof}
\medskip

\section{Generalized Polya-Szeg\"o inequality}
\label{PSineq}

The first main result of the paper is the following

\begin{theorem}
	\label{gpz}
Let $\varrho:\R_+\times\R^N\to\R_+$ be a continuous
function. For any function $u\in W^{1,p}_+(\R^N)$,
let us set
$$
J(u)=\int_{\R^N} \varrho(u,\nabla u)dx.
$$
Moreover, let $(H_n)_{n\geq 1}$ be a dense sequence in the set of closed half spaces containing~$0_{\R^N}$. 
For $u\in W^{1,p}_+(\R^N)$, define a sequence $(u_n)$ by setting
$$
\begin{cases}
u_0=u  &\\
u_{n+1}=u_n^{H_1\ldots H_{n+1}}. &
\end{cases}
$$
Assume that the following conditions hold:
\begin{enumerate}
	\item  $$-\infty<J(u)<+\infty;$$
	\item 
	\begin{equation} 
	\label{monotonicity}
\liminf_n J(u_n)\leq J(u);
    \end{equation}
	\item if $(u_n)$ converges weakly to some $v$ in $W^{1,p}_+(\R^N)$, then 
	$$
	J(v)\leq\liminf_n J(u_n).
	$$
\end{enumerate}
Then
$$
J(u^*)\leq J(u).
$$
\end{theorem}

\begin{proof}
By the (explicit) approximation results contained in~\cite{explicit,JvS},
we know that $u_n\to u^*$ in $L^p(\R^N)$ as $n\to\infty$.
Moreover, by Lemma~\ref{concretelem} applied with $j(s,|\xi|)=|\xi|^p$, we have
\begin{equation}
	\label{uguagplapl}
	\|\nabla u\|_{L^p(\R^N)}=\|\nabla u_0\|_{L^p(\R^N)}=\cdots=\|\nabla u_n\|_{L^p(\R^N)},
	\quad\text{for all $n\in\N$}.
\end{equation}
In particular, up to a subsequence, $(u_n)$ is weakly convergent 
to some function $v$ in $W^{1,p}(\R^N)$. By uniqueness of the weak limit in $L^p(\R^N)$
one can easily check that $v=u^*$, namely $u_n\rightharpoonup u^*$ in $W^{1,p}(\R^N)$.
Hence, using assumption (3) and ~\eqref{monotonicity}, we have 
\begin{equation}
	J(u^*)\leq \liminf_n J(u_n)\leq J(u),
\end{equation}
concluding the proof.
\end{proof}

\begin{remark}
	A quite large class of functionals $J$ which satisfy assumption~\eqref{monotonicity} 
	of the previous Theorem is provided by Lemma~\ref{concretelem}.
\end{remark}

\begin{corollary}
	\label{applgPSineq}
	Let $j:\R_+\times\R_+\to\R_+$ be a function satisfying the assumptions:
	\begin{enumerate}
		\item $j(s,t)$ is continuous;
		\item $j(s,\cdot)$ is convex for all $s\in \R_+$;
		\item $j(s,\cdot)$ is nondecreasing for all $s\in\R_+$.
	\end{enumerate}
	Then, for all function $u\in W^{1,p}_+(\R^N)$ such that
	$$
	\int_{\R^N} j(u,|\nabla u|)dx<\infty,
	$$
	we have
	$$
	\int_{\R^N} j(u^*,|\nabla u^*|)dx\leq \int_{\R^N} j(u,|\nabla u|)dx.
	$$
\end{corollary}
\begin{proof}
	The assumptions on $j$ imply that $\{\xi\mapsto j(s,|\xi|)\}$ is convex
	so that the weak lower semicontinuity assumption of Theorem~\ref{gpz} holds.	
	We refer the reader e.g.\ to the papers~\cite{ioffe1,ioffe2} by A.\ Ioffe for 
	even more general assumptions.
	Also, assumption~\eqref{monotonicity} of Theorem~\ref{gpz} is provided by 
	means of Lemma~\ref{concretelem}.
\end{proof}

\begin{remark}
	In~\cite[Theorem 4.3]{brock}, F.\ Brock proved Corollary~\ref{applgPSineq}
	for Lipschitz functions having compact support. In order to prove the most
	interesting cases in the applications, the inequality has to hold for 
	functions $u$ in $W^{1,p}_+(\R^N)$. This forces him to assume some growth conditions
	of the Lagrangian $j$, for instance to assume that there exists a positive constant
	$K$ and $q\in [p,p^*]$ such that
	$$
	|j(s,|\xi|)|\leq K(s^q+|\xi|^p),\quad\text{for all $s\in\R_+$ and $\xi\in\R^N$}.
	$$
	In our approach, instead, we can include integrands such as 
	$$
	j(s,|\xi|)=\frac{1}{2}(1+s^{2\alpha})|\xi|^p,\quad\text{for all $s\in\R_+$ and $\xi\in\R^N$},
	$$
	for some $\alpha>0$, which have meaningful physical 
	applications (for instance quasi-linear Schr\"odinger equations, see~\cite{liuwang}
	and references therein).
	We also stress that the approach of~\cite{brock} cannot yield the establishment
	of equality cases (see Theorem~\ref{idcases}).
\end{remark}

\begin{theorem}
	\label{idcases}
	In addition to the assumptions of Theorem~\ref{gpz}, assume that 
\begin{equation}
	\label{strict}
\text{$J(u_n)\to J(u^*)$ as $n\to\infty$ implies that $u_n\to u^*$ in $D^{1,p}(\R^N)$ as $n\to\infty$},
\end{equation}
where we recall that $u_n\rightharpoonup u^*$ in $D^{1,p}(\R^N)$. Then
$$
J(u)=J(u^*)\,\,\Longrightarrow\,\, \|\nabla u\|_{L^p(\R^N)} =\|\nabla u^*\|_{L^p(\R^N)}.
$$
\end{theorem}
\begin{proof}
Assume that we have $J(u)=J(u^*)$. Therefore, by assumption~\eqref{monotonicity}, along a subsequence, we obtain
$$
J(u^*)=\lim_n J(u_n)=J(u).
$$
In turn, by assumption, $u_n\to u^*$ in $D^{1,p}(\R^N)$ as $n\to\infty$. Then, taking 
the limit inside equalities~\eqref{uguagplapl}, we conclude the assertion.
\end{proof}

\begin{remark}
	Assume that $\{\xi\mapsto j(s,|\xi|)\}$ is strictly convex for any $s\in\R_+$ and
	there exists $\nu'>0$ such that $j(s,|\xi|)\geq\nu' |\xi|^p$ for all $s\in\R_+$ and $\xi\in\R^N$. 
	Then, in many cases, assumption~\eqref{strict} is fulfilled for $J(u)=\int_{\R^N}j(u,|\nabla u|)dx$. 
	We refer to~\cite[Section 2]{visintin}.
\end{remark}

\begin{remark}
	Equality cases of the type $\|\nabla u\|_{L^p(\R^N)}=\|\nabla u^*\|_{L^p(\R^N)}$
	have been completely characterized  in the breakthrough paper by Brothers and Ziemer~\cite{bzim}. 
\end{remark}

Let us now set
$$
\displaystyle{M={\rm esssup}_{\R^N} u={\rm esssup}_{\R^N} u^*},\qquad
C^*=\{x\in\R^N:\nabla u^*(x)=0\}.
$$

\begin{corollary}
	\label{identitcor}
	Assume that $\{\xi\mapsto j(s,|\xi|)\}$ is strictly convex and there exists a positive constant
	$\nu'$ such that
	$$
	j(s,|\xi|)\geq\nu' |\xi|^p,\quad\text{for all $s\in\R$ and $\xi\in\R^N$}.
$$	
	Moreover, assume that~\eqref{strict} holds and
$$
\int_{\R^N} j(u,|\nabla u|)dx=\int_{\R^N} j(u^*,|\nabla u^*|)dx,\quad
\mu(C^*\cap (u^*)^{-1}(0,M))=0.
$$
Then there exists $x_0\in\R^N$ such that
$$
u(x)=u^*(x-x_0),\quad\text{for all $x\in\R^N$},
$$
namely $u$ is radially symmetric after a translation in $\R^N$.
\end{corollary}
\begin{proof}
	It is sufficient to combine Theorem~\ref{idcases} with~\cite[Theorem 1.1]{bzim}.
\end{proof}
\medskip

\section{Applications to minimization problems}
\label{applsection}

In this section we shall study a minimization problem of the following form
\begin{equation}
	\label{minprob}
T=\inf\big\{J(u):\,\, u\in {\mathcal C}\big\},
\end{equation}
where
\begin{equation*}
{\mathcal C}=\Big\{u\in W^{1,p}(\R^N,\R^m):\text{$G_k(u_k),\, j_k(u_k,|\nabla u_k|)\in L^1$ for 
any $k$,\, $\sum_{k=1}^m \int_{\R^N}G_k(u_k)dx=1$}\Big\},
\end{equation*}
where $J$ is the functional defined, for $u=(u_1,\dots,u_m)$, by
$$
J(u)=\sum_{k=1}^m \int_{\R^N} j_k(u_k,|\nabla u_k|)dx-\int_{\R^N} F(|x|,u_1,\dots,u_m)dx.
$$
Under suitable additional regularity assumptions
on $j_k,F$ and $G_k$ the solutions to~\eqref{minprob} 
yields a nontrivial solution to the system on $\R^N$
\begin{equation*}
	\begin{cases}
-{\rm div}(D_\xi j_k(u_k,|\nabla u_k|))+D_sj_k(u_k,|\nabla u_k|)+\gamma D_sG_k(u_k)=D_{s_k}F(|x|,u_1,\dots,u_m),\\
\noalign{\vskip4pt}
k=1,\dots,m,
\end{cases}
\end{equation*}
for some Lagrange multiplier $\gamma\in\R$.

\subsection{Assumptions on $j_k,F,G_k$}
Before stating the main results of the section, we collect here the assumptions we take.

\subsection{Assumptions on $j_k$}
Let $m\geq 1$, $1<p<N$ and let 
$$
j_k:\R\times \R_+\to\R_+,\quad \text{for $k=1,\dots,m$}
$$ 
be continuous functions, convex and increasing with respect to
the second argument and such that there exist $\nu>0$ and a continuous and increasing
function $\beta_k:\R_+\to\R_+$ with
\begin{equation}
	\label{coerc}
\nu|\xi|^p\leq j_k(s,|\xi|)|\leq \beta_k(|s|)|\xi|^p,\quad \text{for $k=1,\dots,m$},
\end{equation}
for all $s\in\R_+$ and $\xi\in\R^N$. 
We also consider the following assumptions:
\begin{equation}
	\label{opposite}
\text{$j_k(-s,|\xi|)\leq j_k(s,|\xi|)$,
\,\, for all $s\in \R_-$ and all $\xi\in\R^N$}.
\end{equation}
Moreover, there exists $\alpha\geq p$ such that
\begin{equation}
	\label{antimonot}
\text{$j_k(ts,t|\xi|)\leq t^\alpha j_k(s,|\xi|)$
\,\, for all $t\geq 1$, $s\in \R_+$ and $\xi\in\R^N$}.
\end{equation}

\subsection{Assumptions on $F$}
Let us consider a function
$$
F:\R_+\times\R^m\to\R,
$$ 
of variables $(r,s_1,\dots,s_m)$, measurable with respect $r$ and continuous with respect to 
$(s_1,\dots,s_m)\in\R^N$ with $F(r,0,\dots,0)=0$ for any $r$. We assume that 
\begin{align}
& F(r, s + he_i + ke_j)+ F(r, s) \geq F(r, s + he_i)+ F(r, s + ke_j), \label{supermod1}\\
& F(r_1,s + he_i)+ F(r_0,s) \leq F(r_1,s)+ F(r_0,s + he_i),   \label{supermod2}
\end{align} 
for every $i\neq j$, $i,j=1,\dots,m$ where $e_i$ denotes the $i$-th standard basis vector in $\R^m$, $r > 0$,
for all $h,k > 0$, $s =(s_1,\dots,s_m)\in\R^m_+$ and $r_0,r_1$ such that $0 <r_0 <r_1$.
\vskip2pt
The regularity assumptions on $F$ could be further relaxed via the notion of Borel measurability
(see~\cite{burchhaj}). Conditions~\eqref{supermod1}-\eqref{supermod2} are also known as cooperativity 
conditions and, in general, they are necessary conditions for rearrangement 
inequalities to hold (see~\cite{brokhaj}). Moreover, we assume that
\begin{align}
&\limsup_{(s_1,\dots,s_m)\to (0,\dots,0)^+} \frac{F(r,s_1,\dots,s_m)}{\sum\limits_{k=1}^m s_k^{p}}<\infty,\label{zerocv} \\
&\lim_{|(s_1,\dots,s_m)|\to +\infty} \frac{F(r,s_1,\dots,s_m)}{\sum\limits_{k=1}^m s_k^{p+\frac{p^2}{N}}}=0, \label{grothF}
\end{align}
uniformly with respect to $r$. 

There exist $r_0>0$, $\delta>0$, $\mu_k>0$, $\tau_k\in[0,p)$ and
$\sigma_k\in[0,\frac{p(p-\tau_k)}{N})$ such that $F(r,s_1,\dots,s_m)\geq 0$
for $|r|\leq r_0$ and
\begin{equation}
F(r,s_1,\dots,s_m)\geq \sum_{k=1}^m \mu_k r^{-\tau_k}s_k^{\sigma_k+p},
\qquad\text{for $r>r_0$ and $s\in\R^m_+$ with $|s|\leq\delta$}.
\label{zerocvBis} 
\end{equation}
Also,
\begin{equation}
 \lim_{\underset{(s_1,\dots,s_m)\to (0,\dots,0)^+}{r\to+\infty}}
\frac{F(r,s_1,\dots,s_m)}{\sum\limits_{k=1}^m s_k^{p}}=0.
\label{zeroultima} 
\end{equation}
Finally, we consider the following assumptions:
\begin{equation}
	\label{modineq}
	F(r,s_1,\dots,s_m)\leq F(r,|s_1|,\dots,|s_m|),
\end{equation}
for all $r>0$ and $(s_1,\dots,s_m)\in \R^m$ and 
\begin{equation}
	\label{psphere}
	F(r,t s_1,\dots,t s_m)\geq t^{\alpha} F(r,s_1,\dots,s_m),
\end{equation}
for all $r>0$, $t\geq 1$ and $(s_1,\dots,s_m)\in \R^m_+$, where
$\alpha\geq p$ is the values which appears in 
condition~\eqref{antimonot}.

\begin{remark}
We stress that a condition from below on $F$ 
like~\eqref{zerocvBis} was firstly considered by C.A.\ Stuart in~\cite{stuarbbelow}.
\end{remark}

\begin{remark}
	As a variant, in place of the growth assumptions~\eqref{zerocv}-\eqref{grothF}, one could directly
	assume that there exist $m\geq 1$ constants
	$$
	0<\sigma_k<\frac{p^2}{N},\quad k=1,\dots,m,
	$$
	and a positive constant $C$ such that
	$$
	0\leq F(r,s_1,\dots,s_m)\leq C\sum_{k=1}^m s_k^{p}+C\sum_{k=1}^m s_k^{p+\sigma_k},
	$$
	for all $r>0$ and $(s_1,\dots,s_m)\in\R^m_+$. Indeed, conclusion~\eqref{bddconcl} can be reached again
	by slightly modifying the Gagliardo-Nirenberg inequality~\eqref{gagnir}.
\end{remark}

\begin{remark}
For instance, take 
$$
\beta\geq 0,\quad
\tau\in[0,p),\quad 
\sigma\in[0,\textstyle{\frac{p(p-\tau)}{N}}),
$$ 
and a continuous and decreasing function $a:\R_+\to\R_+$ such that
$$
a(|x|)={\mathcal O}\left(|x|^{-\tau}\right)\quad\text{as $|x|\to\infty$}.
$$
Consider the function
	$$
	F(|x|,s_1,\dots,s_m)=\frac{a(|x|)}{p+\sigma}\sum_{k=1}^m |s_k|^{p+\sigma}+
	\frac{2\beta a(|x|)}{p+\sigma}
	\sum_{\overset{i,j=1}{i<j}}^m|s_i|^{\frac{p+\sigma}{2}}|s_j|^{\frac{p+\sigma}{2}}.
	$$
	Hence~\eqref{supermod1}-\eqref{psphere} 
	are fulfilled. This allows to treat elliptic systems of the type
	\begin{equation*}
		\begin{cases}
		-\Delta_p u_k+\gamma u^{p-1}_k=a(|x|)|u_k|^{p+\sigma-2}u_k+ \beta a(|x|)  
		{\displaystyle\sum_{i\neq k}^m}  |u_i|^{\frac{p+\sigma}{2}}|u_k|^{\frac{p+\sigma-4}{2}}u_k,\,   & \text{in $\R^N$},\\
		\quad k=1,\dots,m.
	\end{cases}
	\end{equation*}
	In the particular case $m=2$, $p=2$ and $a(s)=1$ (thus $\tau=0$), the above 
	system reduces to the important class of physical
	systems, systems of weakly coupled Schr\"odinger equations
	\begin{equation*}
		\begin{cases}
		-\Delta u+\gamma u=|u|^{\sigma}u+
		\beta|u|^{\frac{\sigma-2}{2}}|v|^{\frac{\sigma+2}{2}}u,\,\,\,   & \text{in $\R^N$}\\
		-\Delta v+\gamma v=|v|^{\sigma}u+
		\beta|v|^{\frac{\sigma-2}{2}}|u|^{\frac{\sigma+2}{2}}v,\,\,\,  & \text{in $\R^N$},
	\end{cases}
	\qquad 0<\sigma<4/N.
	\end{equation*}
	These problems, particularly in the case where $\sigma=2$ (thus 
	in the range $\sigma<4/N$ only for $N=1$) have been deeply
	investigated in the last few years, mainly with respect to the problem
	of existence of bound and ground state depending on the values of $\beta$
	(see e.g.\ \cite{sirakov} and references therein). 
\end{remark}

\begin{remark}
In general, the upper bound $\sigma\leq p^2/N$ in the growth conditions 
on $F$ is a necessary condition for the minimization 
problem~\eqref{minprob} to be well posed, otherwise $T=-\infty$. In fact, assume that 
$(w_1,\dots,w_m)$ is an element of ${\mathcal C}$.
Then we have that $(w^\delta_1,\dots,w^\delta_m)\in{\mathcal C}$ for all $\delta\in (0,1]$, where
$w^\delta_j(x)=\delta^{-N/p}w_j(x/\delta)$. Hence, taking for instance $j_k$ such that there exists 
a positive constant $C$ with $j_k(s,|\xi|)\leq C|\xi|^p$ and
$$
	F(s_1,\dots,s_m)=\frac{1}{p+\sigma}\sum_{k=1}^m |s_k|^{p+\sigma}+
	\frac{2}{p+\sigma}
	\sum_{\overset{k,h=1}{h<k}}^m|s_h|^{\frac{p+\sigma}{2}}|s_k|^{\frac{p+\sigma}{2}},
	$$
	by a simple change of scale we find
	$$
	T\leq J(w^\delta_1,\dots,w^\delta_m)\leq \frac{C}{\delta^p}\sum_{k=1}^m \int_{\R^N}|\nabla w_k|^pdx
	-\frac{1}{\delta^{\frac{N\sigma}{p}}}\int_{\R^N}F(w_1,\dots,w_m)dx,
	$$
	which, letting $\delta\to 0^+$, yields $T=-\infty$, provided 
	that $\frac{N\sigma}{p}>p$, hence $\sigma>\frac{p^2}{N}$.
	In some cases, instead, $T$ is $-\infty$ for larger values of $\sigma$.
	Consider, for instance, the case
	$$
	j_k(s,|\xi|)=(1+s^{2\alpha_k})|\xi|^p,\quad s\in\R,\,\,\xi\in\R^N.
	$$
	for $\alpha_k>0$, $k=1,\dots,m$. Therefore, after scaling, it follows that
	$$
	\sum_{k=1}^m\int_{\R^N} j_k(w^\delta_k,|\nabla w_k^\delta|)dx\leq 
	\frac{C}{\delta^p}+\sum_{k=1}^m\frac{C'}{\delta^{\frac{2\alpha_kN+p^2}{p}}}
	-\frac{C''}{\delta^{\frac{N\sigma}{p}}},
	$$
	where $C,C',C''$ are positive contants. But then $T=-\infty$ if
	$$
	\frac{N\sigma}{p}>\max_{k=1,\dots,m}\big\{p,\frac{2\alpha_kN+p^2}{p}\big\}=
	\max_{k=1,\dots,m}\frac{2\alpha_kN+p^2}{p},
	$$
	namely $\sigma>2\alpha_{{\rm max}}+p^2/N$,
	where $\alpha_{{\rm max}}=\max\{\alpha_k:k=1,\dots,m\}$. In fact, 
	the presence of powers of $u$ in front of the gradient term $|\nabla u|^p$ 
	allows to recover some regularity on $u$  
	as soon as the functional is finite (see e.g.~\cite{liuwang})
	and improve the growth conditions we assumed for the nonlinearity 
	$F$ at infinity.
\end{remark}

\subsection{Assumptions on $G_k$}
Consider $m\geq 1$ continuous and $p$-homogeneous functions
$$
G_k:\R\to\R_+,\quad G_k(0)=0,\quad \text{for $k=1,\dots,m$}
$$ 
such that there exists $\gamma>0$ such that
\begin{equation}
	\label{Gkass} 
G_k(s)\geq \gamma |s|^p,\quad \text{for all $s\in\R$}.
\end{equation}

\subsection{Statement of the results}

In the above framework, the main results of the section are the following

\begin{theorem}
	\label{mainappl}
	Assume that conditions~\eqref{coerc}-\eqref{grothF} and~\eqref{Gkass} hold.
Then the minimum problem~\eqref{minprob} admits a radially symmetric and radially decreasing
nonnegative solution.
\end{theorem}

For a vector function $(u_1,\dots,u_m)$, let us set
$$
\displaystyle{M_i={\rm esssup}_{\R^N} u_i={\rm esssup}_{\R^N} u^*_i},\qquad
C^*_i=\{x\in\R^N:\nabla u^*_i(x)=0\}.
$$

\begin{theorem}
	\label{mainappl2}
	Assume that conditions~\eqref{coerc}-\eqref{grothF} and~\eqref{Gkass} hold and 
	that the function $\{\xi\mapsto j(s,|\xi|)\}$ is strictly convex and~\eqref{strict} holds. Then, for any nonnegative 
	solution $u$ to problem~\eqref{minprob} such that
	\begin{equation}
		\label{cond*}
	\mu(C^*\cap (u^*_i)^{-1}(0,M_i))=0,\quad\text{for some $i\in\{1,\dots,m\}$}.
\end{equation}
	the component $u_i$ is radially symmetric and radially decreasing, after a suitable
	translation in $\R^N$.
In particular, if~\eqref{cond*} holds for any $i=1,\dots,m$, the solution $u$ 
is radially symmetric and radially decreasing, after suitable translations in $\R^N$.
\end{theorem}

We point out that there are situations where condition~\eqref{opposite}-\eqref{antimonot} can be replaced
by a monotonicity conditions on $j$ with respect to $s$. A well established
sign condition for these type of operators, which is often involved in both {\em existence}
and {\em regularity} questions (see e.g.~\cite{candeg,serrin,squastoul,tolks}) is the following: 
there exists $R\geq 0$ such that
\begin{equation}
	\label{signclassical}
s\frac{\partial j}{\partial s} (s,t)\geq 0,\quad\text{for all $t\in\R_+$ and $s\in\R$ with $|s|\geq R$}.
\end{equation}
There are also counterexamples in the literature showing that, for $j=j(x,u,\nabla u)$, 
if condition~\eqref{signclassical} is {\em not} fulfilled
(for instance if~\eqref{antimonot} is satisfied), the solutions of the Euler-Lagrange equation might be
{\rm unbounded} (see~\cite{fresh}). 
\vskip2pt
In the last two results of this section we provide
the existence and symmetry properties of least energy solutions for a class of quasi-linear elliptic
problems by assuming, among other things, condition~\eqref{signclassical}.
This problem has recently been investigated in~\cite{jjsq} by the second 
author jointly with L.\ Jeanjean via a combination of tools from non-smooth
analysis and recent results on the symmetry properties for homogeneous constrained 
minimization problems (see~\cite{bjm,Ma}). Here we obtain the result
as an application of Corollary~\ref{identitcor}. The prize one has to pay
is that an additional information on the measure of the critical set for the Schwarz rearrangement
of a solution is needed. However this condition on $u^*$ is quite natural and, as it 
is shown in~\cite{bzim}, without this additional
assumption there are counterexample to equality cases. 

\begin{theorem}
	\label{mainappl3}
	Assume that $m=1$, $1<p<N$, $F=0$ and that $G_1=G:\R\to\R$ is
a function of class $C^1$ with $G'=g$ and 
\begin{align*}
& \limsup_{s\to 0}\frac{G(s)}{|s|^{p^*}}\leq 0,  \\
& \lim_{s\to\infty}\frac{g(s)}{|s|^{p^*-1}}=0.  
\end{align*}
	Moreover, assume that $j(s,|\xi|):\R\times\R_+\to\R$ is a function of class $C^1$ in
$s$ and $\xi$ and denote by $j_s$ and $j_t$ the derivatives of
$j$ with respect of $s$ and $t=|\xi|$ respectively. We assume that, for any $s\in\R$,
\begin{equation*}
\text{the map $\{\xi\mapsto
j(s,|\xi|\}$ is strictly convex, increasing and $p$-homogeneous;}
\end{equation*}
Moreover, there exist positive constants $c_1,c_2,c_3,c_4$ and $R$ such that
\begin{equation*}
 c_1|\xi|^p\leq j(s,|\xi|)\leq
c_2|\xi|^p,\qquad\text{for all $s\in\R$ and $\xi\in\R^N$};
\end{equation*}
\begin{equation*}
 |j_s(s,|\xi|)|\leq c_3|\xi|^p,\quad
|j_t(s,|\xi|)|\leq c_4|\xi|^{p-1}, \qquad\text{for all $s\in\R$
and $\xi\in\R^N$};
\end{equation*}
\begin{equation*}
 j_s(s,|\xi|)s\geq 0,\qquad\text{for all $s\in\R$ with
$|s|\geq R$ and $\xi\in\R^N$}.
\end{equation*}
Then equation
\begin{equation}
	\label{eqthm3}
-{\rm div}(D_\xi j(u,|\nabla u|))+j_s(u,|\nabla u|)=g(u),\quad\text{in $\R^N$}
\end{equation}
admits positive, radially symmetric and radially decreasing least energy solutions. 
Furthermore, if~\eqref{strict} holds, any
least energy solution $u$ of~\eqref{eqthm3} such that
\begin{equation}
	\label{meascondd}
	\mu(C^*\cap (u^*)^{-1}(0,M))=0,
\end{equation}
is positive, radially symmetric and radially decreasing, up to a translation in $\R^N$.
\end{theorem}

\medskip

\subsection{Proof of Theorem~\ref{mainappl}}
\begin{proof}
Let $u^h=(u^h_1,\dots,u^h_m)\subset {\mathcal C}$ 
be a minimizing sequence for $J|_{{\mathcal C}}$. Then
\begin{gather}
 \lim_{h}\left(\sum_{k=1}^m \int_{\R^N} j_k(u_k^h,|\nabla u_k^h|)dx-\int_{\R^N} F(|x|,u_1^h,\dots,u_m^h)dx\right)=T, \label{min1}\\
 G_k(u^h_k),\, j_k(u_k^h,|\nabla u_k^h|)\in L^1(\R^N), \quad
\sum_{k=1}^m\int_{\R^N}G_k(u^h_k)dx=1,\quad\text{for all $h\in\N$}.\notag
\end{gather}
Taking into account assumption~\eqref{modineq}, we have
$$
F(|x|,u_1^h,\dots,u_m^h)\leq F(|x|,|u_1^h|,\dots,|u_m^h|),\quad\text{for all $h\in\N$}.
$$
Moreover, by the fact that $|\nabla u_k^h(x)|=|\nabla |u_k^h(x)||$
for a.e.\ $x\in\R^N$, for all $k=1,\dots,m$ and $h\in\N$, 
in light of assumption~\eqref{opposite}, it holds
$$
j_k(|u_k^h|,|\nabla |u_k^h||)\leq j_k(u_k^h,|\nabla u_k^h|),\quad\text{for all $k=1,\dots,m$ and $h\in\N$}.
$$
In conclusion, we have
$$
J(|u^h_1|,\dots,|u^h_m|)\leq J(u^h_1,\dots,u^h_m),\quad\text{for all $h\in\N$},
$$
so that we may assume, without loss of generality, that $u^h_k\geq 0$ a.e.,
for all $k=1,\dots,m$ and $h\in\N$.
Let us now prove that $(u^h)$ is bounded in $W^{1,p}(\R^N,\R^m)$. Indeed, as $(u^h)\subset {\mathcal C}$,
by assumption~\eqref{Gkass} on $G_k$, it follows that the sequence $(u^h)$ is uniformly bounded in $L^p(\R^N)$.
By combining the growth assumptions~\eqref{zerocv}-\eqref{grothF}, 
for every $\eps>0$ there exists $C_\eps>0$ such that
\begin{equation}
	\label{subcrit}
F(r,s_1,\dots,s_m)\leq C_\eps\sum_{k=1}^m s_k^p+\eps \sum_{k=1}^m s_k^{p+\frac{p^2}{N}},\quad
\text{for all $r,s_1,\dots,s_m\in (0,\infty)$}.	
\end{equation}
Therefore, in view of the Gagliardo-Nirenberg inequality
\begin{equation}
	\label{gagnir}
\|u^h_k\|_{L^{p+\frac{p^2}{N}}(\R^N)}^{p+\frac{p^2}{N}}\leq C\|u^h_k\|_{L^p(\R^N)}^{\frac{p^2}{N}}
\|\nabla u^h_k\|_{L^p(\R^N)}^{p},\quad
\text{for all $h\in\N$},
\end{equation}
for all $\eps>0$ there exists $C_\varepsilon>0$ such that, for all $h\in\N$,
\begin{align*}
\int_{\R^N} F(|x|,u_1^h,\dots,u_m^h)dx & \leq C_\eps\sum_{k=1}^m \|u^h_k\|_{L^{p}(\R^N)}^p
+\eps\sum_{k=1}^m \|u^h_k\|_{L^{p+\frac{p^2}{N}}(\R^N)}^{p+\frac{p^2}{N}} \\
& \leq CC_\eps+C\eps\sum_{k=1}^m \|u^h_k\|_{L^p(\R^N)}^{\frac{p^2}{N}}\|\nabla u^h_k\|_{L^p(\R^N)}^{p} \\
& \leq CC_\eps+C\varepsilon\sum_{k=1}^m \|\nabla u^h_k\|_{L^p(\R^N)}^p.
\end{align*}
In turn, by combining assumption~\eqref{coerc} with~\eqref{min1}, fixed $\eps_0\in(0,\frac{\nu}{C})$, it follows
\begin{equation}
	\label{bddconcl}
(\nu-C\eps_0)\sum_{k=1}^m\|\nabla u^h_k\|_{L^p(\R^N)}^p\leq CC_{\eps_0}+T,\qquad\text{for all $h\in\N$}.
\end{equation}
yielding the desired boundedness of $(u^h)$ in $W^{1,p}(\R^N,\R^m)$. Hence,
after extracting a subsequence, which we still denote by $(u^h)$, we get for any $k=1,\dots,m$
\begin{equation}
\label{subseq}
u^h_k\rightharpoonup u_k\,\,\, \text{in $L^{p^*}(\R^N)$},\,\,
Du^h_k\rightharpoonup Du_k\,\,\, \text{in $L^{p}(\R^N)$},\,\,
u^h_k(x)\to u_k(x)\quad\text{a.e.\ $x\in\R^N$}.
\end{equation}
For any $k=1,\dots,m$ and $h\in\N$, let us denote by $u^{*h}_k$ the Schwarz symmetric rearrangement
of $u^h_k$. By means of~\cite[Theorem 1]{burchhaj}, we have
\begin{equation}
\label{nonlincontrol}
\int_{\R^N} F(|x|,u_1^h,\dots,u_m^h)dx\leq \int_{\R^N} F(|x|,u_1^{*h},\dots,u_m^{*h})dx.
\end{equation}
Moreover, by Corollary~\ref{applgPSineq}, we have
$$
\int_{\R^N} j_k(u^{*h}_k,|\nabla u^{*h}_k|)dx\leq \int_{\R^N} j_k(u^h_k,|\nabla u^h_k|)dx,
$$
Finally, as it is well-known, we have 
$$
G_k(u^{*h}_k),\, j_k(u_k^{*h},|\nabla u_k^{*h}|)\in L^1(\R^N),\quad
\sum_{k=1}^m\int_{\R^N}G_k(u^{*h}_k)dx=1,\quad\text{for all $h\in\N$}.
$$
Hence, since
$$
J(u^{*h})\leq J(u^h),\quad u^{*h}\in{\mathcal C},\quad\text{for all $h\in\N$},
$$
it follows that $u^{*h}=(u^{*h}_1,\dots,u^{*h}_m)$ is also a positive minimizing 
sequence for $J|_{{\mathcal C}}$, which is now radially symmetric and radially decreasing.
In what follows, we shall denote it back to $u^{h}=(u^{h}_1,\dots,u^{h}_m)$. Taking into account that $u^{h}_k$ 
is bounded in $L^{p}(\R^N)$, it follows that (see \cite[Lemma A.IV]{BL1})
\begin{equation}
	\label{decay}
u^h_k(x)\leq c_k|x|^{-\frac{N}{p}},\qquad\text{for all $x\in\R^N\setminus\{0\}$ and $h\in\N$},
\end{equation}
for a positive constant $c_k$, independent of $h$. In turn, 
by virtue of condition~\eqref{zeroultima}, for all
$\eps>0$ there exists $\rho_\eps>0$ such that
$$
|F(|x|,u_1^h(|x|),\dots,u_m^h(|x|))|\leq\eps\sum_{k=1}^m |u_k^h(|x|)|^{p},
\quad\text{for all $x\in\R^N$ with $|x|\geq \rho_\eps$},
$$
yielding, via the boundedness of $(u^h_k)$ in $L^{p}(\R^N)$,
\begin{equation}
	\label{inzerocontrollo}
\Big|\int_{\R^N\setminus B(0,\rho_\eps)}F(|x|,u_1^h,\dots,u_m^h)dx\Big|
\leq\eps\sum_{k=1}^m \|u_k^h\|^{p}_{p}\leq \eps C.
\end{equation}
Since~\eqref{decay} holds also for the pointwise limit $(u_1,\dots,u_m)$, analogously it follows
\begin{equation*}
\Big|\int_{\R^N\setminus B(0,\rho_\eps)}F(|x|,u_1,\dots,u_m)dx\Big|\leq \eps C.
\end{equation*}
On the other hand, by the growth assumption~\eqref{grothF} and the local strong convergence of $(u^h)$
to $u$ in $L^m$ with $m<p^*$, for this $\rho_\eps$ we obtain
$$
\lim_h \int_{B(0,\rho_\eps)} F(|x|,u_1^h,\dots,u_m^h)dx =\int_{B(0,\rho_\eps)} F(|x|,u_1,\dots,u_m)dx.
$$
Then, by~\eqref{inzerocontrollo}, we have
\begin{equation}
\label{conFtoF}
\lim_h\int_{\R^N} F(|x|,u_1^h,\dots,u_m^h)dx =\int_{\R^N} F(|x|,u_1,\dots,u_m)dx.
\end{equation}
Also as $j(s,t)$ is positive, convex and increasing in the $t$-argument (and thus $\xi\mapsto j(s,|\xi|)$
is convex), by well known lower semicontinuity results (cf.\ \cite{ioffe1,ioffe2}, 
see e.g.~\cite[Theorem 3.23]{dacorog}), 
for any $k=1,\dots,m$ it follows
\begin{equation}
\label{lowersem}
\int_{\R^N}j_k(u_k,|Du_k|)dx\leq \liminf_{h}\int_{\R^N}j_k(u^h_k,|Du^h_k|)dx,
\end{equation}	
where the right hand side is uniformly bounded, in view of~\eqref{min1} and~\eqref{conFtoF}.
Hence, in conclusion we have $j_k(u_k,|Du_k|)\in L^1(\R^N)$ for any $k=1,\dots, m$ and 
\begin{equation}
	\label{lowerSC}
J(u)\leq \liminf_{h} J(u^h)=\lim_{h} J(u^h)=T.
\end{equation}
Then, to conclude the proof, it is sufficient to show that 
the limit $u$ satisfies the constraint.
Let us first prove that $T<0$. For any $\theta\in(0,1]$, let us consider the function
$$
\Upsilon^\theta_k(x)=
\frac{\theta^{N/p^2}}{d_k^{1/p}}e^{-\theta|x|^p},\quad
d_k=m\int_{\R^N} G_k(e^{-|x|^p})dx,\quad k=1,\dots,m.
$$
Therefore $(\Upsilon^\theta_1,\dots,\Upsilon^\theta_m)$ belongs to ${\mathcal C}$ with 
$\Upsilon^\theta_k\in L^\infty(\R^N,\R_+)$ for all $k=1,\dots,m$ since
by the $p$-homogeneity of any $G_k$ and a simple change of scale we get
$$
\sum_{k=1}^m\int_{\R^N}G_k(\Upsilon^\theta_k(x))dx=
\sum_{k=1}^m\frac{\theta^{N/p}}{d_k}\int_{\R^N}G_k(e^{-\theta|x|^p})dx=
\sum_{k=1}^m\frac{1}{d_k}\int_{\R^N}G_k(e^{-|x|^p})dx=1.
$$
Notice that
$$
|\nabla \Upsilon^\theta_k(x)|^p=p^p\frac{\theta^{N/p+p}}{d_k}e^{-p\theta|x|^p}|x|^{p(p-1)}
,\quad x\in\R^N,\quad k=1,\dots,m.
$$
Recalling that the function $\beta_k$ is continuous, we have
$$
\Lambda_k=\sup_{x\in\R^N}\sup_{\theta\in[0,1]}\beta_k(\Upsilon^\theta_k(x))<\infty.
$$
By virtue of the growth condition~\eqref{coerc} and a simple change of variable, it follows that
\begin{align*}
\int_{\R^N} j_k(\Upsilon^\theta_k(x),|\nabla \Upsilon^\theta_k(x)|)dx  &\leq \int_{\R^N} \beta_k(\Upsilon^\theta_k(x))|\nabla \Upsilon^\theta_k(x)|^pdx 
\leq \Lambda_k\int_{\R^N} |\nabla \Upsilon^\theta_k(x)|^pdx \\
&\leq \frac{\Lambda_k p^p\theta^{N/p+p}}{d_k}\int_{\R^N}e^{-p\theta|x|^p}|x|^{p(p-1)}dx
=\theta C_k,
\end{align*}
where we have set
$$
C_k=\frac{\Lambda_k p^p}{d_k}\int_{\R^N}e^{-p|x|^p}|x|^{p(p-1)}dx,\quad k=1,\dots,m.
$$
In light of assumption~\eqref{zerocvBis}, 
since of course $0\leq \Upsilon^\theta_k(x)\leq \theta^{N/p^2}/{d_k^{1/p}}\leq \delta$ 
for $\theta$ sufficiently small and all $k=1,\dots,k$, we obtain
\begin{align*}
\int_{\R^N} F(|x|,\Upsilon^\theta_1(x),\dots,\Upsilon^\theta_m(x))dx & \geq
\sum_{k=1}^m\frac{\mu_k}{d_k^{\frac{\sigma_k+p}{p}}}
\theta^{\frac{N(\sigma_k+p)}{p^2}} \int_{\{|x|\geq r_0\}}  |x|^{-\tau_k}
e^{-\theta(\sigma_k+p)|x|^p}dx \\
&\geq \sum_{k=1}^m\theta^{\frac{N\sigma_k+p\tau_k}{p^2}} C_k',
\end{align*}
where we have set
$$
C_k'=\frac{\mu_k}{d_k^{\frac{\sigma_k+p}{p}}}\int_{\{|x|\geq r_0\}}  |x|^{-\tau_k}
e^{-(\sigma_k+p)|x|^p}dx,
\quad k=1,\dots,m.
$$
In conclusion, collecting the previous inequalities, for $\theta>0$ sufficiently small, 
\begin{align*}
T &\leq \sum_{k=1}^m \int_{\R^N} j_k(\Upsilon^\theta_k(x),|\nabla \Upsilon^\theta_k(x)|)dx   
 -\int_{\R^N} F(|x|,\Upsilon^\theta_1(x),\dots,\Upsilon^\theta_m(x))dx \\
& \leq \theta \sum_{k=1}^m \Big(C_k-\theta^{\frac{N\sigma_k+p\tau_k-p^2}{p^2}} C_k'\Big)<0,
\end{align*}
as $N\sigma_k+p\tau_k-p^2<0$, yielding the desired assertion. Now, of course, we have
$$
\sum_{k=1}^m\int_{\R^N}G_k(u_k)dx\leq\liminf_{h\to\infty}\sum_{k=1}^m\int_{\R^N}G_k(u^{h}_k)dx=1.
$$
In particular it holds $G_k(u_k)\in L^1(\R^N)$, for every $k=1,\dots,m$. Notice also that 
we have $(u_1,\dots,u_m)\neq (0,\dots,0)$, otherwise we would get a contradiction
by combining inequality~\eqref{lowerSC} with $T<0$.
Choosing the positive number
$$
\tau:=\Big(\sum_{k=1}^m\int_{\R^N}G_k(u_k)dx\Big)^{-1/p}\geq 1,
$$
via the $p$-homogeneity of $G_k$ it follows that $(\tau u_1,\dots,\tau u_m)$ belongs ${\mathcal C}$ as
$$
\sum_{k=1}^m\int_{\R^N}G_k(\tau u_k)dx=\tau^p\sum_{k=1}^m\int_{\R^N}G_k(u_k)dx=1.
$$
Therefore, by taking into account conditions~\eqref{antimonot} and~\eqref{psphere}, it follows 
from~\eqref{lowerSC} that
\begin{align*}
T &\leq \sum_{k=1}^m \int_{\R^N} j_k(\tau u_k,\tau |\nabla u_k|)dx-\int_{\R^N} F(|x|,\tau u_1,\dots,\tau u_m)dx \\
 &\leq \tau^\alpha\Big(\sum_{k=1}^m \int_{\R^N} j_k(u_k,|\nabla u_k|)dx-\int_{\R^N} F(|x|,u_1,\dots,u_m)dx\Big) \\
& =\tau^\alpha J(u)\leq \tau^\alpha T.
\end{align*}
This, being $T<0$, yields $\tau=1$ so that that $(u_1,,\dots,u_m)\in {\mathcal C}$,
concluding the proof.
\end{proof}

\begin{remark}
	Assume that the map
	\begin{equation}
		\label{strictconv}
	\Big\{\xi\mapsto \sum_{k=1}^m j_k(s_k,|\xi_k|)\Big\}
\end{equation}
	is strictly convex and there exists $\nu>0$ such that
	$$
	\sum_{k=1}^m j_k(s_k,|\xi_k|)\geq \nu \sum_{k=1}^m |\xi_k|^p,
	\quad\text{for all $s\in\R^m$ and $\xi\in\R^{mN}$}.
	$$
	From the proof of Theorem~\ref{mainappl}
	we know that the weak limit $(u_1,\dots,u_m)$ of the minimizing sequence 
satisfies the constraint. Then, recalling~\eqref{conFtoF} we have
\begin{align*}
T&=\sum_{k=1}^m \int_{\R^N} j_k(u_k^h,|\nabla u_k^h|)dx-\int_{\R^N} F(|x|,u_1^h,\dots,u_m^h)dx+o(1) \\
&=  \int_{\R^N} \sum_{k=1}^m(j_k(u_k^h,|\nabla u_k^h|)-j_k(u_k,|\nabla u_k|))dx \\
&+\sum_{k=1}^m \int_{\R^N} j_k(u_k,|\nabla u_k|)dx-\int_{\R^N} F(|x|,u_1,\dots,u_m)dx+o(1)\\
&\geq  T+\int_{\R^N} \sum_{k=1}^m(j_k(u_k^h,|\nabla u_k^h|)-j_k(u_k,|\nabla u_k|))dx+o(1),
\end{align*}
as $h\to\infty$. Taking into account the weak 
lower semicontinuity, along a subsequence,
$$
\int_{\R^N} \sum_{k=1}^m j_k(u_k^h,|\nabla u_k^h|)dx=
\int_{\R^N} \sum_{k=1}^m j_k(u_k,|\nabla u_k|)dx+o(1),\quad\text{as $h\to\infty$}.
$$
In turn, by the strict convexity of~\eqref{strictconv}, whenever a condition as~\eqref{strict} holds, we have
$$
\lim_h\sum_{k=1}^m \|\nabla u_k^h\|_p^p=\sum_{k=1}^m \|\nabla u_k\|^p_p.
$$
Recalling~\eqref{Gkass}, since we have $G_k(u_k^h)\to G_k(u_k)$ a.e.~as $h\to\infty$,
$$
\sum_{k=1}^m\int_{\R^N}G_k(u_k^h)dx=\sum_{k=1}^m\int_{\R^N}G_k(u_k)dx
\quad
\text{and}
\quad
\gamma\sum_{k=1}^m |u^h_k|^p\leq \sum_{k=1}^m G_k(u_k^h),
$$
we conclude that $(u_1^h,\dots,u_m^h)$ converges strongly to $(u_1,\dots,u_m)$
in $W^{1,p}(\R^N,\R^m)$.
\end{remark}
\medskip

\subsection{Proof of Theorem~\ref{mainappl2}}
\begin{proof}
We know from Theorem~\ref{mainappl} that problem~\eqref{minprob} admits at least
a radially symmetry and radially decreasing positive solution $u$. Assume now that $v=(v_1,\dots,v_m)$ is another 
positive solution to problem~\eqref{minprob}. 
Hence, if $v^{*}_k$ denotes the Schwarz symmetrization of $v_k$, 
as $\int_{\R^N} G_k(v_k^*)dx=\int_{\R^N} G_k(v_k)dx$ for all $k=1,\dots,m$, we have
\begin{align*}
T & \leq \sum_{k=1}^m \int_{\R^N}j_k(v_k^*,|\nabla v_k^*|)dx-\int_{\R^N} F(|x|,v_1^*,\dots,v_m^*)dx \\
 &\leq \sum_{k=1}^m \int_{\R^N}j_k(v_k,|\nabla v_k|)dx-\int_{\R^N} F(|x|,v_1,\dots,v_m)dx =T.
\end{align*}
In turn, equality must holds, and since
\begin{align*}
\int_{\R^N}j_k(v_k^*,|\nabla v_k^*|)dx &\leq\int_{\R^N}j_k(v_k,|\nabla v_k|)dx,\quad\text{for all $k=1,\dots,m$}, \\
\int_{\R^N} F(|x|,v_1,\dots,v_m)dx &\leq \int_{\R^N} F(|x|,v_1^*,\dots,v_m^*)dx,
\end{align*}
in particular one has 
$$
\int_{\R^N}j_k(v_k^*,|\nabla v_k^*|)dx=\int_{\R^N}j_k(v_k,|\nabla v_k|)dx,
\quad\text{for all $k=1,\dots,m$}.
$$
Therefore, in light of Corollary~\ref{identitcor}, any component $u_i$ of the solution which satisfies
$\mu(C^*\cap (u^*_i)^{-1}(0,M_i))=0$ is automatically radially symmetric and radially decreasing, 
after suitable translations in $\R^N$.
\end{proof}
\medskip

\subsection{Proof of Theorem~\ref{mainappl3}}
\begin{proof}
It is sufficient to follow the proof of~\cite[Lemma 2.11]{jjsq} to show that
there exists a positive solution $u$ of the minimization problem
\begin{equation}
	\label{lastminpb}
\min\Big\{\int_{\R^N}j(u,|\nabla u|)dx: u\in D^{1,p}(\R^N),\, G(u)\in L^1(\R^N),\,\int_{\R^N}G(u)dx=1\Big\}.
\end{equation}
Since $\int_{\R^N}G(u^*)dx=1$ and $\int_{\R^N} j(u^*,|\nabla u^*|)dx\leq \int_{\R^N} j(u,|\nabla u|)dx$
by Corollary~\ref{applgPSineq}, one can assume that
$u$ is radially symmetric and radially decreasing. By Corollary~\ref{identitcor} it follows that every 
solution $w$ of the problem (which is positive by~\cite[Proposition 5]{bjm})
which satisfies $\mu(C^*\cap (w^*)^{-1}(0,M))=0$ is radially symmetric 
and radially decreasing after a suitable translation in $\R^N$.
By means of~\cite[Lemma 1]{bjm} (more precisely a simple generalization
of the lemma to cover general operators $j(u,|\nabla u|)$ which are $p$-homogeneous
in the gradient), minimizers of~\eqref{lastminpb} and least energy solutions to~\eqref{eqthm3} 
correspond via scaling. In order to apply~\cite[Lemma 1]{bjm}, one also needs that the properties
${\bf (C_2)}$ and ${\bf (C_3)}$ indicated therein are satisfied. This properties have been proved in
\cite[Lemma 2.13 and Lemma 2.16]{jjsq}. This concludes the proof of Theorem~\ref{mainappl3}. 
\end{proof}

\bigskip
\bigskip
\noindent
{\bf Acknowledgments.} The authors wish to thank Louis Jeanjean 
for carefully reading the manuscript and for providing useful suggestions 
which helped to improve the paper. Furthermore, the first author is grateful 
to Almut Burchard for illuminating discussions concerning rearrangement techniques  
during his stay at the University of Virginia in 2004.

\bigskip
\bigskip

\end{document}